\documentclass[12pt]{article}
\usepackage[margin=1.3in]{geometry}
\usepackage{amsthm,fullpage}
\usepackage{amsmath}
\usepackage{amssymb}
\usepackage{mathtools}
\usepackage{float}
\usepackage{algorithm}
\usepackage[noend]{algpseudocode}
\usepackage[toc,page]{appendix}
\usepackage{graphicx}
\usepackage[hidelinks]{hyperref}
\usepackage[pdftex,dvipsnames,table]{xcolor}
\usepackage[export]{adjustbox} % for adjusting images
\usepackage{subfiles}
\usepackage{enumerate}
\usepackage{subcaption}

\usepackage{makecell}

\theoremstyle{plain}
\numberwithin{equation}{section}

\newtheorem{theorem}{Theorem}[section]
\newtheorem{corollary}[theorem]{Corollary}
\newtheorem{lemma}[theorem]{Lemma}
\newtheorem{conjecture}[theorem]{Conjecture}

\newtheorem{remark}{Remark}[section]

\hypersetup{
	colorlinks=True,
	linkcolor= blue,
 	citecolor=blue
 }
\definecolor{ag}{rgb}{0.55, 0.71, 0.0}

\DeclarePairedDelimiterX{\inprod}[2]{\langle}{\rangle}{#1, #2}
\usepackage{xargs}
\usepackage[colorinlistoftodos,prependcaption,textsize=tiny]{todonotes}
\newcommandx{\unsure}[2][1=]{\todo[linecolor=red,backgroundcolor=red!25,bordercolor=red,#1]{#2}}
\newcommandx{\change}[2][1=]{\todo[linecolor=blue,backgroundcolor=blue!25,bordercolor=blue,#1]{#2}}
\newcommandx{\info}[2][1=]{\todo[linecolor=OliveGreen,backgroundcolor=OliveGreen!25,bordercolor=OliveGreen,#1]{#2}}
\newcommandx{\improvement}[2][1=]{\todo[linecolor=Plum,backgroundcolor=Plum!25,bordercolor=Plum,#1]{#2}}
\title{Bollob\'{a}s - Nikiforov Conjecture for graphs with not so many triangles}
\author{Hitesh Kumar, Shivaramakrishna Pragada}

\date{\today}

\begin{document}

\maketitle

\begin{abstract}
 Bollob\'{a}s and Nikiforov \cite{bollobas2007cliques} conjectured that for any graph $G \neq K_n$ with $m$ edges
 \[ \lambda_1^2+\lambda_2^2\le \bigg( 1-\frac{1}{\omega(G)}\bigg)2m\]
 where $\lambda_1$ and $\lambda_2$ denote the two largest eigenvalues of the adjacency matrix $A(G)$, and $\omega(G)$ denotes the clique number of $G$. This conjecture was recently verified for triangle-free graphs \cite{lineigenvalues} and regular graphs \cite{zhang2024regular}. In \cite{elphick2024square_sum_conjec}, a generalization of this conjecture was proposed. In this note, we verify this generalized conjecture for the family of graphs on $m$ edges, which contain at most $O(m^{1.5-\varepsilon})$ triangles for some $\varepsilon > 0$. In particular, we show that the conjecture is true for planar graphs, book-free graphs and cycle-free graphs.
\end{abstract}



\begin{center}
     \emph{In memory of Prof. Vladimir Nikiforov.}
 \end{center}

\section{Introduction}
For a simple graph $G$ on $n$ vertices and $m$ edges, let $A(G)$ denote its \emph{adjacency matrix}. We denote the \emph{eigenvalues} of $A(G)$ by 
\[\lambda_1\geq \lambda_2\geq \cdots \geq \lambda_{n-1}\geq \lambda_n.\] Let $(n^+, n^0, n^-)$ denote the \emph{inertia} of $G$, i.e. $G$ has $n^+$ positive eigenvalues, eigenvalue $0$ with multiplicity $n^0$, and $n^-$ negative eigenvalues. We denote the \emph{clique number} of $G$ by $\omega(G)$. Let $t(G)$ denote the number of triangles in $G$. Let $s_k(G)=\sum_{i=1}^k \lambda_i^2$, i.e. $s_k(G)$ is the square sum of first $k$ eigenvalues of $G$.  For a graph $G$ define \[\Lambda_k(G)= \frac{s_k(G)}{m}.\] 

In 2002, Nikiforov \cite{nikiforov2002spectralturan} showed the following result, which implies the classical Tur\'{a}n's Theorem on cliques.
\begin{theorem}
For every graph $G$, we have 
    \[\Lambda_1(G)  \leq 2\bigg(1 - \frac{1}{\omega(G)}\bigg).\]
\end{theorem}

Later, in 2007,  Bollob\'{a}s and Nikiforov \cite{bollobas2007cliques} proposed that the following stronger conjecture might be true. 
\begin{conjecture}[Bollob\'{a}s-Nikiforov Conjecture]\label{conj:BN}
For every graph $G \neq K_{n}$, we have 
\[\Lambda_2(G) \leq 2\bigg(1 - \frac{1}{\omega(G)}\bigg).\]
\end{conjecture}

The above conjecture has been verified for triangle-free graphs by Lin, Ning and Wu \cite{lineigenvalues} and for regular graphs by Zhang \cite{zhang2024regular}. Elphick, Linz and Wocjan \cite{elphick2024square_sum_conjec}, after some computational investigation, suggested that the following generalization might be true. 
\begin{conjecture}\label{conj:BN_general}
For every graph $G$,  
\[\Lambda_\ell (G) \leq 2\bigg(1 - \frac{1}{\omega(G)}\bigg)\]
where $\ell=\min\{n^+, \omega(G)\}$.
\end{conjecture}

Elphick, Linz and Wocjan \cite{elphick2024square_sum_conjec} verified the conjecture for some graph families, including weakly perfect graphs and Kneser graphs. Also, see the recent survey by Liu and Ning \cite{liu_unsolved_2023} for progress on these and related conjectures.

In this note, we verify the above conjectures in a strong way for graphs with not so many triangles. To make our claims precise, we introduce some notation. For positive constants $\varepsilon$ and $c$ let $\mathcal{G}(\varepsilon, c)$ denote the family of graphs $G$ with $t(G)\leq c m^{1.5-\varepsilon}$. Let $\mathcal{G}_\omega(\varepsilon, c)$ denote the subfamily of graphs in $\mathcal{G}(\varepsilon, c)$ whose clique number equals fixed constant $\omega$. Now, we can state our main result.

\begin{theorem}\label{thm:main1}
Let $\omega \ge 3$. Let $G \in \mathcal{G}_\omega(\varepsilon, c)$ with $m$ edges and let $k\geq 1$. If $m \geq \big(2.2c\omega^{2k}\big)^{1/\varepsilon}$ then 
\[\Lambda_\ell(G) < 2\bigg(\frac{\sqrt[3]{\omega}}{1+ \sqrt[3]{\omega}} + \frac{1}{\omega^{k}}\bigg)\]
where $\ell = \min\{n^+,\omega\}$.
\end{theorem}

Note that taking $k=3$ in the above theorem implies Conjecture \ref{conj:BN_general} for graphs in $\mathcal{G}_\omega(\varepsilon, c)$ with sufficiently many edges. We emphasize that our bound outperforms the conjectured bound for any fixed $\omega$ and large $m$. In other words, there is no hope for equality in Conjecture \ref{conj:BN_general} for graphs with not so many triangles. 

Since there are many interesting graph classes with not so many triangles, our result expands the set of all known graph classes for which Conjecture \ref{conj:BN_general} is true. In particular, it is true for planar graphs, $B_k$-free graphs, and $C_k$-free graphs where $B_k$ denotes the book on $k+2$ vertices and $C_k$ denotes the $k$-cycle. Recall that a graph $G$ is said to be $H$-free if $G$ does not contain $H$ as a subgraph. 

\begin{figure}[H]
\begin{subfigure}{0.6\textwidth}
    \centering
\begin{tikzpicture}[scale=0.85]
\draw  (0,1)-- (0,-1);
\draw  (2,0)-- (0,1);
\draw  (2,0)-- (0,-1);
\draw  (3,0)-- (0,1);
\draw  (3,0)-- (0,-1);
\draw  (6,0)-- (0,1);
\draw  (6,0)-- (0,-1);

\draw [fill=black] (0,1) circle (1.5pt);
\draw [fill=black] (0,-1) circle (1.5pt);
\draw [fill=black] (2,0) circle (1.5pt);
\draw [fill=black] (3,0) circle (1.5pt);
\draw [fill=black] (6,0) circle (1.5pt);
\draw [fill=black] (3.75,0) circle (1.5pt);
\draw [fill=black] (4.25,0) circle (1.5pt);
\draw [fill=black] (4.75,0) circle (1.5pt);

\draw (2.1,0.3) node {$1$};
\draw (3.1,0.3) node {$2$};
\draw (6.1,0.3) node {$k$};
\end{tikzpicture}
    \subcaption{Book $B_k$}
\end{subfigure}
\begin{subfigure}{0.4\textwidth}
    \centering
\begin{tikzpicture}
\draw  (0,1)-- (0,-1);
\draw  (-1.56,0)-- (0,1);
\draw  (-1.56,0)-- (0,-1);
\draw  (0,1)-- (1.64,0);
\draw  (1.64,0)-- (0,-1);

\draw [fill=black] (0,1) circle (1.5pt);
\draw [fill=black] (0,-1) circle (1.5pt);
\draw [fill=black] (-1.56,0) circle (1.5pt);
\draw [fill=black] (1.64,0) circle (1.5pt);
\end{tikzpicture}
    \subcaption{Diamond $B_2$}
\end{subfigure}    
\end{figure}

\begin{corollary}\label{cor:graph_classes}
Conjecture \ref{conj:BN_general} holds for the following graph classes:
\begin{enumerate}[$(i)$]
    \item Planar graphs with at least $405$ edges.
    \item Outerplanar graphs with at least $76$ edges.
    \item $B_k$-free graphs with at least $\big(\frac{10.06(k-1)\sqrt{k+1}}{3}\big)^{2}$ edges. In particular, the conjecture holds for diamond-free graphs with at least $34$ edges. 
    \item $C_k$-free graphs with at least $\big(\frac{10.06(k-3)\sqrt{k}}{3}\big)^{2}$ edges.
\end{enumerate}
\end{corollary}

In this note, we also observe that the bound in the Bollob\'{a}s-Nikiforov Conjecture can be improved for graphs with not so many triangles. The upper bound in Conjecture \ref{conj:BN} approaches 2 as $\omega \rightarrow \infty$. Our bound, given below, approaches 1.  

\begin{theorem}\label{thm:main2}
Let $G \in \mathcal{G}(\varepsilon, c)$. Then
     \[\Lambda_2(G) \leq 1 + o(1). \]
\end{theorem}

In conclusion, our results suggest that the bottleneck in proving Conjectures \ref{conj:BN} and \ref{conj:BN_general} is when the graph $G$ has many triangles, i.e. when $t(G)=\Omega(m^{1.5}).$ 

\section{Proof of Theorem \ref{thm:main1}}
 We first recall a consequence of Holder's inequality. For $x = (x_1,x_2,\dots,x_n) \in \mathbb{R}^n$ and $1 \leq p < \infty$, define $\Vert x \Vert_p = \Big(\sum_{i=1}^{n} \vert x_i \vert^p \Big)^\frac{1}{p}$. 
\begin{lemma}\label{p_norm_ineq}
Let $x$ be a vector in $\mathbb{R}^n$ and $1 \leq p\leq q$. Then
 \[ ||x||_q \leq ||x||_p \leq n^{\frac{1}{p} -\frac{1}{q}} ||x||_q.\]
\end{lemma}

Using the above lemma, we estimate the number of triangles in a graph in terms of its eigenvalues.
\begin{lemma}\label{lemma:triangle}
Let $t(G)$ denote the number of triangles in a graph $G$. Let $s_k=\sum_{i=1}^k \lambda_i^2$. Then, 
\[ 6t(G)\ge \frac{1}{\sqrt{k}} \big(s_k\big)^{3/2} -  \big(2m - s_k \big)^{3/2}, \]
for all $1\le k\le n^+$.
\end{lemma}
\begin{proof}
Consider the vectors $x = (\lambda_1^2, \ldots, \lambda_{k}^2)$ and $y = (\lambda_n^2, \lambda_{n-1}^2,\dots, \lambda^2_{n-n^{-}+1})$. By Lemma \ref{p_norm_ineq} with $p=2$ and $q=3$, we obtain that
\begin{equation*}
     \sum_{i=1}^{k}\lambda_i^3  \geq \frac{1}{\sqrt{k}} \big(s_k\big)^{3/2},
\end{equation*}
and
\begin{equation*}
     \sum_{i = n-n^{-} +1}^n |\lambda_i|^3 \leq \bigg( \sum_{i = n-n^{-} +1}^n \lambda_i^2 \bigg)^{3/2} \leq \big(2m - s_k \big)^{3/2}. 
\end{equation*}
It follows that 
\[ 6t(G) = \sum_{i=1}^n \lambda_i^3 \geq \sum_{i=1}^{k}\lambda_i^3 - \sum_{i = n-n^- +1}^n |\lambda_i|^3 \ge \frac{1}{\sqrt{k}} \big(s_k\big)^{3/2} -  \big(2m - s_k\big)^{3/2}. 
\qedhere
\]
\end{proof}

We are ready to prove Theorem \ref{thm:main1}. Our proof is via contradiction. It relies on the fact that if the square sum of eigenvalues is too large, then the number of triangles in the graph has to be $\Omega(m^{1.5})$, contrary to what we assumed. 

\begin{proof}[Proof of Theorem \ref{thm:main1}]
Let $\delta=\dfrac{\Lambda_{\ell}}{2}$. Then, by Lemma \ref{lemma:triangle}, 
\begin{gather*}
6cm^{1.5-\varepsilon} \ge 6t(G) \ge \frac{1}{\sqrt{\ell}}  \big(s_\ell \big)^{3/2} - \big(2m - s_\ell \big)^{3/2} 
 \geq (2m)^{3/2}\bigg(\frac{\delta^{3/2}}{\sqrt{\omega}} - \big(1 - \delta\big)^{3/2}\bigg),
\end{gather*}
since $\ell \le \omega$.  Suppose that Theorem \ref{thm:main1} does not hold. Then $\delta \geq \big(\frac{\sqrt[3]{\omega}}{1+ \sqrt[3]{\omega}} + \frac{1}{\omega^k}\big)$, and we have
\begin{align*}
  6cm^{1.5-\varepsilon}  & \geq (2m)^{3/2}\bigg[\frac{1}{\sqrt{\omega}}\bigg(\frac{\sqrt[3]{\omega}}{1+ \sqrt[3]{\omega}} + \frac{1}{\omega^k}\bigg)^{3/2} - \bigg(\frac{1}{1+ \sqrt[3]{\omega}} - \frac{1}{\omega^k}\bigg)^{3/2}\bigg] \\
  & > (2m)^{3/2} \bigg[\frac{1}{\sqrt{\omega}}\bigg(\frac{\sqrt[3]{\omega}}{1+ \sqrt[3]{\omega}}\bigg)^{3/2}+ \frac{1}{\sqrt{\omega}}\bigg( \frac{1}{\omega^k}\bigg)^{3/2} - \bigg(\frac{1}{1+ \sqrt[3]{\omega}}\bigg)^{3/2}\bigg] \\
& > (2m)^{3/2} \frac{1}{\omega^{1.5k + 0.5}}.
\end{align*}
The above inequalities hold since $\omega\ge 3$ and $\omega^k\ge \omega \ge \sqrt[3]{\omega} +1$. This gives us $m < (2.2c\omega^{2k})^{1/\varepsilon}$ which is a contradiction. 
\end{proof}

\begin{remark}\label{remark:edge_lower_bound}
In the proof above, if $\delta > (1-\frac{1}{\omega})$ then
\begin{align*}
  6cm^{1.5-\varepsilon}  & > (2m)^{3/2}\bigg(\frac{1}{\sqrt{\omega}}\bigg(1-\frac{1}{\omega}\bigg)^{3/2} - \bigg(\frac{1}{\omega}\bigg)^{3/2}\bigg) \\
& = \frac{(2m)^{3/2}}{\sqrt{\omega}}\bigg(\bigg(1-\frac{1}{\omega}\bigg)^{3/2} - \frac{1}{\omega}\bigg) \\  
& > \frac{(2m)^{3/2}}{3\sqrt{3}\sqrt{\omega}}\big(2\sqrt{2} - \sqrt{3}\big).
\end{align*}
The function $\big(1-\frac{1}{\omega}\big)^{3/2} - \frac{1}{\omega}$ is increasing in $\omega$, so the last inequality follows by taking $\omega=3$. This gives us $m < (10.06c\sqrt{\omega})^{1/\varepsilon}$. We conclude that if $G\in \mathcal{G}_\omega(\varepsilon, c)$ has $m\ge (10.06c\sqrt{\omega})^{1/\varepsilon}$ edges then Conjecture \ref{conj:BN_general} holds for $G$. 
\end{remark}

We use the above remark to prove Corollary \ref{cor:graph_classes}. We require the following lemma, which is not difficult to prove.
\begin{lemma}\label{lemma:triangle_counting} For every graph $G$, 
\[t(G) =  \frac{1}{3}\sum_{v\in V(G)} m(G[N(v)]),\]
where $m(G[N(v)])$ denotes the number of edges in the subgraph induced by $N(v)$.
\end{lemma}

We finish this section with a proof of Corollary \ref{cor:graph_classes}.
\begin{proof}[Proof of Corollary \ref{cor:graph_classes}] The following claims are well known. We give the proofs for completeness using Lemma \ref{lemma:triangle_counting}. 
\begin{enumerate}[$(i)$]
    \item If $G$ is a planar graph then for every vertex $v \in V(G)$, the induced subgraph $G[N(v)]$ is outerplanar and has at most $2\deg(v)-3$ edges. This gives $t(G)\le \frac{4m-3n}{3}$ and since $m\le 3n-6$ it follows that $t(G) \leq m-2$. As planar graphs are $K_5$-free, $\omega(G)\le 4$. 
    \item If $G$ is an outerplanar graph then for every vertex $v \in V(G)$, the induced subgraph $G[N(v)]$ is acyclic and has at most $\deg(v)-1$ edges. This gives $t(G)\le \frac{2m-n}{3}$. Since $m\le 2n-3$, it follows that $t(G) \leq \frac{m-1}{2}$. As outerplanar graphs are $K_4$-free, $\omega(G)\le 3$. 
    \item If $G$ is a $B_k$-free graph then for every vertex $v \in V(G)$, the induced subgraph $G[N(v)]$ is $K_{1,k}$-free and has at most $\deg(v)\frac{(k-1)}{2}$ edges. Thus $t(G)\le \frac{(k-1)m}{3}$ and clearly $\omega(G) \leq k+1$.
    \item If $G$ is a $C_k$-free graph then for every vertex $v \in V(G)$, the induced subgraph $G[N(v)]$ is $P_{k-1}$-free and by Erd\"{o}s-Gallai Theorem \cite{erdos_gallai_1959}, it  has at most $\deg(v)\frac{(k-3)}{2}$ edges. Thus $t(G)\le \frac{(k-3)m}{3}$ and clearly $\omega(G) \leq k$.
\end{enumerate}
Using the above claims and Remark \ref{remark:edge_lower_bound}, the Corollary \ref{cor:graph_classes} is immediate. 
\end{proof}

\section{Proof of Theorem \ref{thm:main2}}
In \cite{lineigenvalues}, Li, Ning and Wu proved Conjecture \ref{conj:BN} for triangle-free graphs and characterized the graphs which attain equality. In \cite[Remark 4.1]{signedspecturan2023}, Kannan and Pragada gave an upper bound on spectral radius in terms of the number of triangles in the graph. In the following theorem, we obtain a bound for $s_2 =\lambda_1^2 + \lambda_2^2$ in terms of number of triangles in the graph. This generalizes the above-mentioned results. 

\begin{theorem}\label{thm:triangle_bound}
For every graph $G$ with $m\ge 2$ edges, 
\[\lambda_1^2 + \lambda_2^2 \leq m + \big(3t(G)\big)^{2/3}\]
and the inequality is strict if $t(G) > 0 $.
\end{theorem}

The proof of the above theorem relies on majorization. To that end, we first recall some definitions and results. For a vector $x \in \mathbb{R}^n$, let $x^\downarrow$ denote the vector obtained by rearranging the entries of $x$ in the non-increasing order. Given two vectors $x, y \in \mathbb{R}^n $, we say that  $x$ is \textit{weakly majorized} by $y$, denoted by $x \prec_w y$, if
\begin{gather*}
    \sum_{j=1}^{k} {x_j}^\downarrow \leq \sum_{j=1}^{k} {y_j}^\downarrow
\end{gather*}
for all $1 \leq k \leq n.$
If $x \prec_w y \ \text{and}  \ \sum_{i=1}^{n} x_i^\downarrow = \sum_{i=1}^{n} y_i^\downarrow,$
then we say that $x$ is \textit{majorized} by $y$ and denoted by $x \prec y$.

\begin{theorem}[{\cite[Theorem 2.1]{lineigenvalues}}]\label{rev_majorization_thm}
    Let $x = (x_1,x_2,\dots,x_n),\ y = (y_1,y_2,\dots,y_n) \in \mathbb{R}^n_{\geq 0}$, such that $x_i$ and $y_i$ are in non-increasing order. If $y \prec_w x$, then 
    \[\Vert y\Vert_p \leq \Vert x\Vert_p\]
    for any real number $p> 1$, and equality holds if and only if $x=y$.
\end{theorem}

We now give a proof of Theorem \ref{thm:triangle_bound}.
\begin{proof}[Proof of Theorem \ref{thm:triangle_bound}]
The statement is true for complete graphs $K_n$ for $n\ge 3$. So assume $G\neq K_n$ and hence $\lambda_2\ge 0$. Now if $\lambda_1^2 + \lambda_2^2 \leq m$, then we are done. So suppose $\lambda_1^2 + \lambda_2^2 > m$. Let $\lambda_1^2 + \lambda_2^2 = m + \delta$, for some $\delta>0$.  Since $\sum_{i=1}^n \lambda_i^2 = 2m$, we have
\[ 2(\lambda_1^2 + \lambda_2^2) =  2m + 2\delta = \sum_{i=1}^n \lambda_i^2 + 2\delta \]
which implies
\begin{equation}\label{eq:1}
\lambda_1^2 + \lambda_2^2 = \sum_{i=3}^n \lambda_i^2 + 2\delta \geq \sum_{i=n-n^{-}+1}^n\lambda_i^2 + 2\delta.
\end{equation}

Take $x = (\lambda_1^2, \lambda_2^2,0,\dots,0)^T$ and $y = (\lambda_n^2, \lambda_{n-1}^2,\dots, \lambda^2_{n-n^{-}+1},\delta,\delta)^T$ in $\mathbb{R}^{n^{-}+2}$. We claim that $y \prec_w x$. Clearly, $\lambda_1^2\ge \max\{\lambda_n^2, \delta\}$ by \eqref{eq:1}. Again by \eqref{eq:1}, $\lambda_1^2 + \lambda_2^2 \geq \sum_{i=n-n^{-}+1}^n\lambda_i^2 + 2\delta$. Thus $y \prec_w x$. Moreover, $x \neq y$ because $y$ has at least three non-zero entries. 

Using Theorem \ref{rev_majorization_thm} with $p=\frac{3}{2}$, we get
$$\Vert x \Vert^{3/2}_{3/2} > \Vert y \Vert^{3/2}_{3/2}, $$ that is,
$$ \lambda_1^3 +  \lambda_2^3 > \sum_{i = n-n^-+1}^n|\lambda_i|^3 + 2\delta^{3/2}.$$
This implies that
\begin{align*}
    6t(G) &= \sum_{i=1}^n \lambda_i^3 \geq \lambda_1^3 + \lambda_2^3 - \sum_{i = n-n^-+1}^n|\lambda_i|^3\\ &> 2\delta^{3/2} = 2\big(\lambda_1^2 + \lambda_2^2- m\big)^{3/2}.
\end{align*}
Upon rearrangement we get the desired inequality.
\end{proof}

Now if $G\in \mathcal{G}(\varepsilon, c)$ then $t(G)\le cm^{1.5-\varepsilon}$. By Theorem \ref{thm:triangle_bound} we get 
\[\lambda_1^2 + \lambda_2^2 \leq m + \big(3cm^{1.5-\varepsilon}\big)^{2/3}.\] Theorem \ref{thm:main2} is now immediate.

\section*{Acknowledgement}
The authors thank Prof. Bojan Mohar for his helpful comments.

	% -------------References-----------------
	\addcontentsline{toc}{section}{\textsf{References}}
	\bibliographystyle{plain}
	\bibliography{references}

\vspace{2cm} 
Hitesh Kumar\\
Email: {\tt hitesh\_kumar@sfu.ca}

Shivaramakrishna Pragada\\
Email: {\tt shivaramakrishna\_pragada@sfu.ca}

\textsc{Department of Mathematics, Simon Fraser University, Burnaby, BC \ V5A 1S6, Canada}

\end{document}